\documentclass{article}[12pt]
\usepackage[sf]{titlesec}

\usepackage{amssymb}
\usepackage{latexsym}
\usepackage{amsmath, amsfonts, amsthm}
\usepackage{xspace}
\usepackage{multicol}
\usepackage{colordvi}
\usepackage{epsf}
\usepackage{graphicx}
\usepackage{color,psfrag} 
\usepackage{xcolor}
\usepackage{graphics}
\usepackage{amsthm}
\usepackage{amscd}
\usepackage{sidecap}
\usepackage{wrapfig}
\usepackage{subfig}
\usepackage{enumitem}

\newtheoremstyle{slplain}
{.1\baselineskip\@plus.2\baselineskip\@minus.2\baselineskip}
{.1\baselineskip\@plus.2\baselineskip\@minus.2\baselineskip}
{}
{}
{\bfseries}
{.}
{ }
{}

\theoremstyle{slplain}

\titlespacing{\section}{0pt}{*0.45}{*0.45}
\titlespacing{\subsection}{0pt}{*0.45}{*0.45}
\titlespacing{\subsubsection}{0pt}{*0.45}{*0.45}

\newcommand{\la}{\left \langle}
\newcommand{\ra}{\right\rangle}

\newcommand{\norm}[1]{\left\lVert #1 \right\rVert}

\newtheorem{theorem}{Theorem}[section]
\newtheorem{corollary}[theorem]{Corollary}
\newtheorem{lemma}[theorem]{Lemma}

\theoremstyle{definition}

\newtheorem{assumption}[theorem]{Assumption}

\theoremstyle{remark}
\newtheorem{remark}[theorem]{Remark}

\numberwithin{equation}{section}

\usepackage{mathtools}

\DeclarePairedDelimiter\floor{\lfloor}{\rfloor}

\title{Scaling Limits of Neural Networks with the Xavier Initialization and Convergence to a Global Minimum}

\author{Justin Sirignano\footnote{Department of Industrial \& Systems Engineering, University of Illinois at Urbana-Champaign, Urbana, E-mail: jasirign@illinois.edu} \phantom{.}  and Konstantinos Spiliopoulos\footnote{Department of Mathematics and Statistics, Boston University, Boston, E-mail: kspiliop@math.bu.edu}
\thanks{K.S. was partially supported by the National Science Foundation (DMS 1550918)}\\
}

\begin{document}

\maketitle



\begin{abstract}
We analyze single-layer neural networks with the Xavier initialization in the asymptotic regime of large numbers of hidden units and large numbers of stochastic gradient descent training steps. The evolution of the neural network during training can be viewed as a stochastic system and, using techniques from stochastic analysis, we prove the neural network converges in distribution to a random ODE with a Gaussian distribution. The limit is completely different than in the typical mean-field results for neural networks due to the $\frac{1}{\sqrt{N}}$ normalization factor in the Xavier initialization (versus the $\frac{1}{N}$ factor in the typical mean-field framework). Although the pre-limit problem of optimizing a neural network is non-convex (and therefore the neural network may converge to a local minimum), the limit equation minimizes a (quadratic) convex objective function and therefore converges to a global minimum. Furthermore, under reasonable assumptions, the matrix in the limiting quadratic objective function is positive definite and thus the neural network (in the limit) will converge to a global minimum with zero loss on the training set.

\end{abstract}

\section{Introduction}

Consider a single-layer neural network with the Xavier initialization:

\begin{eqnarray}
g^{N}(x; \theta) = \frac{1}{\sqrt{N}} \sum_{i=1}^N C^i \sigma( W^i \cdot  x),\nonumber
\end{eqnarray}
where $C^i \in \mathbb{R}$, $W^i \in \mathbb{R}^d$, $x \in \mathbb{R}^d$, and $\sigma(\cdot): \mathbb{R} \rightarrow \mathbb{R}$.  The number of hidden units is $N$ and the output is scaled by a factor $\frac{1}{\sqrt{N}}$ (the widely-used Xavier initialization, see \cite{Xavier}).

The objective function is
\begin{eqnarray}
\mathcal{L}^{N}(\theta) = \frac{1}{2}\mathbb{E} \bigg{[} (Y - g^N(X; \theta) )^2 \bigg{]},\nonumber
\end{eqnarray}
where the data $(X,Y) \sim \pi(dx,dy)$, $Y \in \mathbb{R}$, and the parameters $\theta = (C^1, \ldots, C^N, W^1, \ldots, W^N ) \in \mathbb{R}^{N \times (1 + d) }$. For notational convenience, we may refer to $g^N(x; \theta)$ as $g^N(x)$ in our analysis below.

The model parameters $\theta$ are trained using stochastic gradient descent. The parameter updates are given by:

\begin{eqnarray}
C^i_{k+1} &=& C^i_k + \frac{ \alpha^N }{\sqrt{N}} ( y_k - g_k^N(x_k) ) \sigma(W^i_k \cdot x_k ), \notag \\
W^i_{k+1} &=& W^i_k + \frac{ \alpha^N}{\sqrt{N}}  ( y_k - g_k^N(x_k) ) C^i_k \sigma'(W^i_k \cdot x_k ) x_k, \notag \\
g_k^N(x) &=&  \frac{1}{\sqrt{N}} \sum_{i=1}^N C^i_k \sigma( W^i_k \cdot x),
\label{SGDupdates}
\end{eqnarray}
for $k = 0, 1, \ldots, T N$ where $T > 0$. $\alpha^N$ is the learning rate (which may depend upon $N$). The data samples are $(x_k, y_k)$ are i.i.d. samples from a distribution $\pi(dx,dy)$.

We impose the following assumption.
\begin{assumption} \label{A:Assumption1} We have that
\begin{itemize}
\item The activation function $\sigma\in C^{2}_{b}(\mathbb{R})$, i.e. $\sigma$ is twice continuously differentiable and bounded.

\item The randomly initialized parameters $(C_0^i, W_0^i)$ are i.i.d., mean-zero random variables with a distribution $ \mu_0(dc, dw)$.
\item The random variable $C_0^i$ has compact support and $\la \norm{w}, \mu_0 \ra < \infty$.
\item The sequence of data samples $(x_k, y_k)$ is i.i.d. from the probability distribution $\pi(dx,dy)$.
\item There is a fixed dataset of $M$ data samples $(x^{(i)}, y^{(i)} )_{i=1}^M$ and therefore $\pi(dx,dy) = \displaystyle \frac{1}{M} \sum_{i=1}^M \delta_{(x^{(i)}, y^{(i)}) }(dx,dy)$.

\end{itemize}
\end{assumption}
Note that the last assumption also implies that $\pi(dx,dy)$ has compact support.

We will study the limiting behavior of the network output $g^N_k(x)$ for $x \in \mathcal{D} = \{ x^{(1)}, \ldots, x^{(M)} \}$ as the number of hidden units $N$ and stochastic gradient descent steps $TN$ simultaneously become large. The network output converges in distribution to the solution of a random ODE as $N \rightarrow \infty$.

\subsection{Main Results} \label{MainResults}

Define the empirical measure
\begin{eqnarray}
\nu_k^N = \frac{1}{N} \sum_{i=1}^N \delta_{C^i_k, W^i_k}.\nonumber
\end{eqnarray}

Note that the neural network output can be written as the inner-product
\begin{eqnarray}
g_k^N(x) =  \la c \sigma(w \cdot x), \sqrt{N} \nu_k^N \ra.\nonumber
\end{eqnarray}

Due to Assumption \ref{A:Assumption1}, as $N \rightarrow \infty$ and for $x \in \mathcal{D}$,
\begin{eqnarray}
g_0^N(x) \overset{d} \rightarrow \mathcal{G}(x),
\label{Gdefinition}
\end{eqnarray}
where $\mathcal{G} \in \mathbb{R}^M$ is a Gaussian random variable. We also of course have that

\begin{eqnarray}
\nu_0^N \overset{p} \rightarrow \nu_0 \equiv \mu_0.\nonumber
\end{eqnarray}

Define the scaled processes
\begin{eqnarray}
h_t^N &=& g_{\floor*{N t}}^N, \notag \\
\mu_t^N &=& \nu_{\floor*{N t}}^N,\nonumber
\end{eqnarray}
where $g_k^N = \bigg{(} g_k^N(x^{(1)}), \ldots, g_k^N(x^{(M)}) \bigg{)}$, $h_t^N(x) = g_{\floor*{N t}}^N(x)$, and $h_t^N = \bigg{(} h_t^N(x^{(1)}), \ldots, h_t^N(x^{(M)}) \bigg{)}$.

We will study convergence in distribution of the random process $(\mu_t^N, h_t^N)$ as $N \rightarrow \infty$ in the space $D_E([0,T])$ where $E = \mathcal{M}(\mathbb{R}^{1+d}) \times \mathbb{R}^M $.  $D_E([0,T])$ is the Skorokhod space and $\mathcal{M}(S)$ is the space of probability measures on $S$.

The main contribution of our work is a rigorous proof that a neural network with the Xavier initialization and trained with stochastic gradient descent converges in distribution to a random ODE as the number of units and training steps become large. In addition, our convergence analysis will also address several interesting questions:
\begin{itemize}
\item Our results provide a rigorous convergence guarantee for the Xavier initialization (i.e., the $\frac{1}{\sqrt{N}}$ normalization factor), which is almost universally used in deep learning models. \emph{A priori} it is unclear if the neural network $g_k^N(x)$ will converge as $N \rightarrow \infty$ since, for $k > 0$, the $C^i \sigma( W^i \cdot  x)$ is correlated with $C^j \sigma( W^j  \cdot x)$ and therefore a limit may not exist. If a limit did not exist, this would imply that the neural network model could have poor numerical behavior for large $N$. We prove that a limit does exist.
\item Although the pre-limit problem of optimizing a neural network is non-convex (and therefore the neural network may converge to a local minimum), the limit equation minimizes a quadratic objective function.
\item We show that the matrix in the limiting quadratic objective function is positive definite, and therefore the neural network (in the limit) will converge to a global minimum with zero loss on the training set.
\end{itemize}

Convergence to a global minimum for a neural network has been recently proven in \cite{JasonLee}, \cite{Du1}, and \cite{UCLA2018}. Our work contributes to this growing literature by showing that convergence to a global minimum is a simple consequence of the mean-field limit for neural networks. A detailed discussion of these papers and other related literature is provided in Section \ref{LitReview}. \\

Our main results are presented below. \\

\begin{theorem} \label{MainTheorem1}
The process $(\mu_t^N, h_t^N)$ converges in distribution in the space $D_E([0,T])$ as $N \rightarrow \infty$ to $(\mu_t, h_t)$ which satisfies, for every $f \in C_2^b( \mathbb{R}^{1+d})$, the random ODE

\begin{eqnarray}
h_t(x) &=& h_0(x) +   \alpha \int_{\mathcal{X} \times \mathcal{Y}} ( y - h_t(x') ) \la \sigma(w \cdot x ) \sigma(w \cdot x'), \mu_t \ra  \pi(dx',dy) dt \notag \\
&+& \alpha \int_{\mathcal{X} \times \mathcal{Y}} ( y - h_t(x') )   \la c^2  \sigma'(w \cdot x')  \sigma'(w \cdot  x) x\cdot x', \mu_t \ra \pi(dx',dy) dt, \notag \\
h_0(x) &=& \mathcal{G}(x), \notag \\
\la f, \mu_t \ra &=& \la f, \mu_0 \ra.
\label{EvolutionEquationIntroductionXavier}
\end{eqnarray}

\end{theorem}
\begin{proof}
See Sections \ref{RelativeCompactness}, \ref{Identification}, and \ref{ProofOfConvergence}.
\end{proof}

Recall that $\mathcal{G} \in \mathbb{R}^M$ is a Gaussian random variable; see equation (\ref{Gdefinition}). In addition, note that $\bar \mu_t$ in the limit equation (\ref{EvolutionEquationIntroductionXavier}) is a constant, i.e. $\mu_t = \mu_0$ for $t \in [0,T]$. Therefore, (\ref{EvolutionEquationIntroductionXavier}) reduces to

\begin{eqnarray}
h_t(x) &=& h_0(x) +   \alpha  \int_0^t \int_{\mathcal{X} \times \mathcal{Y}} ( y - h_s(x') ) \la \sigma(w \cdot x ) \sigma(w \cdot x'), \mu_0 \ra  \pi(dx',dy) ds \notag \\
&+& \alpha \int_0^t \int_{\mathcal{X} \times \mathcal{Y}} ( y - h_s(x') )   \la c^2  \sigma'(w \cdot x')  \sigma'(w \cdot x) x\cdot x', \mu_0 \ra \pi(dx',dy) ds, \notag \\
h_0(x) &=& \mathcal{G}(x).
\label{LimitEquation2}
\end{eqnarray}

Since (\ref{LimitEquation2}) is a linear equation in $C_{\mathbb{R}^M}([0,T])$, the solution $h_t$ is unique. \\

To better understand (\ref{LimitEquation2}), define the matrix $A \in \mathbb{R}^{M \times M}$ where
\begin{eqnarray}
A_{x, x'} = \frac{\alpha}{M} \la \sigma(w \cdot x ) \sigma(w \cdot x'), \mu_0 \ra   +  \frac{\alpha}{M}    \la c^2  \sigma'(w \cdot x')  \sigma'(w \cdot x) x\cdot x', \mu_0 \ra,\nonumber
\end{eqnarray}
where $x, x' \in \mathcal{D}$. $A$ is finite-dimensional since we fixed a training set of size $M$ in the beginning.

Then, (\ref{LimitEquation2}) becomes
\begin{eqnarray}
d h_t &=& A  \bigg{(} \hat Y -  h_t \bigg{)} dt, \notag \\
h_0 &=& \mathcal{G},\nonumber
\end{eqnarray}
where $\hat Y = ( y^{(1)}, \ldots,  y^{(M)} )$.

Therefore, $h_t$ is the solution to a continuous-time gradient descent algorithm which minimizes a quadratic objective function.
\begin{eqnarray}
\frac{d h_t}{d t} &=& - \frac{1}{2} \nabla_h J(\hat Y, h_t), \notag \\
J( y, h) &=& \big{(} y -  h \big{)}^{\top} A \big{(} y -  h \big{)} \bigg{]},  \notag \\
h_0 &=& \mathcal{G}.
\nonumber
\end{eqnarray}

Therefore, even though the pre-limit optimization problem is non-convex, the neural network's limit will minimize a quadratic objective function.

An interesting question is whether $h_t \rightarrow \hat Y$ as $t \rightarrow \infty$. That is, in the limit of large numbers of hidden units and many training steps, does the neural network model converge to a global minimum with zero training error. Theorem \ref{ZeroTrainingError} shows that $h_t \rightarrow \hat Y$ as $t \rightarrow \infty$ if $A$ is positive definite. Corollary \ref{PositiveDefinite} proves that, under reasonable hyperparameter choices and if the data samples are in distinct directions (see \cite{yIto}), $A$ will be positive definite. \\

\begin{theorem} \label{ZeroTrainingError}
If $A$ is positive definite, then
\begin{eqnarray}
h_t \rightarrow \hat Y \phantom{....} \textrm{as} \phantom{....} t \rightarrow \infty.\nonumber
\end{eqnarray}

\end{theorem}
\begin{proof}
Consider the transformation $\tilde h_t = h_t - \hat Y$. Then,
\begin{eqnarray}
d \tilde h_t &=& - A \tilde h_t dt, \notag \\
\tilde h_0 &=& \mathcal{G} - \hat Y.
\nonumber
\end{eqnarray}
Then, $\tilde h_t \rightarrow 0$ (and consequently $h_t \rightarrow \hat Y$) as $t \rightarrow \infty$ if $A$ is positive definite.

\end{proof}

\begin{corollary} \label{PositiveDefinite}
Assume Assumption \ref{A:Assumption1}. A sufficient condition for $A$ to be positive definite is $\sigma(\cdot)$ is non-polynomial and slowly increasing (i.e., $\lim_{x\rightarrow\infty}\frac{\sigma(x)}{x^a}=0$ for every $a>0$), $\mu_{0}$ is positive when evaluated on sets of positive Lebesgue measure and the data samples $x^{(i)}$ are in distinct directions (as defined on page 192 of \cite{yIto}).
\end{corollary}

\begin{proof}
See Section \ref{ProofOFPositiveDefinite}.

\end{proof}

Examples of activation units $\sigma(\cdot)$ satisfying the conditions in Corollary \ref{PositiveDefinite} include sigmoid functions and hyperbolic tangent functions. 
Using a normal distribution for the initialization of the parameters in the neural network is a common choice in practice (covered by the requirements of Corollary \ref{PositiveDefinite}). \\

\begin{remark}
For presentation purposes we have not explicitly denoted the bias term in the model. However, it is clear that this can be handled by requiring the first component of the vector $x$ to be equal to one for example. This would result in the neural network taking the form
$g^{N}(x; \theta) = \frac{1}{\sqrt{N}} \sum_{i=1}^N C^i \sigma( W^i \cdot  x+ b^{i})$. We leave the rest of the details to the interested reader.
\end{remark}

\subsection{Literature Review} \label{LitReview}

\cite{Montanari}, \cite{SirignanoSpiliopoulosNN1}, \cite{SirignanoSpiliopoulosNN2}, and \cite{RVE} study the asymptotics of single-layer neural networks with a $\frac{1}{N}$ normalization; that is, $g^{N}(x; \theta) = \frac{1}{N} \sum_{i=1}^N C^i \sigma( W^i  \cdot x)$. \cite{SirignanoSpiliopoulosNN3} studies the asymptotics of deep (i.e., multi-layer) neural networks with a $\frac{1}{N}$ normalization in each hidden layer. In the single layer case, the limit for the neural network satisfies a partial differential equation. As discussed in \cite{Montanari}, it is \emph{not} necessarily true that the limiting equation (a PDE in this case) will converge to the global minimum of an objective function with zero training error.

The $\frac{1}{N}$ normalization studied in \cite{Montanari}, \cite{SirignanoSpiliopoulosNN1}, \cite{SirignanoSpiliopoulosNN2}, and \cite{RVE} is convenient since the single-layer neural network is then in a traditional mean-field framework where it can be described via an empirical measure of the parameters. However, the $\frac{1}{\sqrt{N}}$ normalization that we study in this paper is more widely-used in practice (it is referred to as the Xavier initialization and was first introduced in \cite{Xavier}). The $\frac{1}{\sqrt{N}}$ normalization requires different analysis than the standard mean-field analysis $\frac{1}{N}$, and it produces a completely different limit. Importantly, under reasonable conditions, the limit equation converges to a global minimum with zero training error. In addition, for the limit to hold, we show that the $\frac{1}{\sqrt{N}}$ normalization requires the effective learning rate for the parameters to be of the order $N^{-3/2}$.

Convergence to a global minimum for a neural network has been recently proven in \cite{JasonLee}, \cite{Du1}, and \cite{UCLA2018}. Although it has been long understood that neural networks have universal approximation properties (see \cite{Hornik1}, \cite{Hornik2}, and \cite{Hornik3}), it has until recently been commonly believed that training algorithms for neural networks (e.g., gradient descent) may converge to a local minimum (and not a global minimum) since neural networks are non-convex.  \cite{JasonLee}, \cite{Du1}, and \cite{UCLA2018} showed that neural networks (under suitable conditions) will converge to a global minimum during training. This result is quite remarkable considering the  optimization problem is non-convex, and it provides an important mathematical guarantee for the field of deep learning.

\cite{JasonLee}, \cite{Du1}, and \cite{UCLA2018} do not study the mean-field limit of a neural network with the Xavier initialization, which is the focus of our paper. Once the mean field limit is established, we show that convergence to a global minimum is a simple consequence of the limit equation. There are also some differences between our assumptions and the assumptions required for the theorems of \cite{JasonLee}, \cite{Du1}, and \cite{UCLA2018}. \cite{JasonLee} and \cite{Du1} study gradient descent while our paper studies stochastic gradient descent, which introduces additional technical challenges due to the stochastic dynamics. \cite{UCLA2018} studies stochastic gradient descent for a framework where the neural network's output layer parameters are not trained. In their paper, the $C^i$ parameters are randomly generated and then frozen (i.e., they do not change during training). In practice, all of the parameters in the neural network, including the output layer parameters, are trained with stochastic gradient descent and therefore it is worthwhile to consider the more general case. \cite{UCLA2018} also imposes an assumption that the loss function vanishes at infinity. \cite{JasonLee}, \cite{Du1}, and \cite{UCLA2018} all require that every data sample has the same magnitude, i.e. $ \norm{ x^{(i)} } =1$ for every $i =1, \ldots, M$. We do not require this assumption.

\cite{NTK} proved a limit for neural networks with a Xavier initialization when they are trained with continuous-time gradient descent. Our paper proves a limit for neural networks trained with the (standard) discrete-time stochastic gradient descent algorithm which is used in practice. Our method of proof is also different than the approach of \cite{NTK}. Whereas \cite{NTK} begins their analysis in continuous time (due to their framework being continuous-time gradient descent), our paper rigorously passes from discrete time (where the stochastic gradient descent updates evolve) to continuous time through weak convergence analysis of appropriate stochastic processes and measure-valued processes. In \cite{NTK}, the authors directly study the evolution of the derivatives of the output with respect to the parameters, while we  address the limiting behavior of the underlying associated stochastic processes and  measure-valued processes.

\subsection{Organization of Paper} \label{Organization}
Section \ref{PreLimit} derives equations describing the evolution of the pre-limit process $(\mu^N, h^N)$. Relative compactness of the family of processes $(\mu^N, h^N)$ is proven in Section \ref{RelativeCompactness}. Section \ref{Identification} proves that any limit point of the process must satisfy the equation (\ref{EvolutionEquationIntroductionXavier}). These results are collected together in Section \ref{ProofOfConvergence} to prove that $(\mu^N, h^N)$ converges in distribution to the solution of equation (\ref{EvolutionEquationIntroductionXavier}). Corollary \ref{PositiveDefinite} is proven in Section \ref{ProofOFPositiveDefinite}.

\section{Evolution of the Pre-limit Process} \label{PreLimit}
We begin by analyzing evolution of the network output $g_k^N(x)$. Using a Taylor expansion,

\begin{eqnarray}
g_{k+1}^N(x)  &=& g_{k}^N(x)  + \frac{1}{\sqrt{N}} \sum_{i=1}^N C^i_{k+1} \sigma( W^i_{k+1} \cdot x) - \frac{1}{\sqrt{N}} \sum_{i=1}^N C^i_k \sigma( W^i_k \cdot x) \notag \\
&=& g_{k}^N(x)  + \frac{1}{\sqrt{N}} \sum_{i=1}^N \bigg{(} C^i_{k+1} \sigma( W^i_{k+1} \cdot x) -  C^i_k \sigma( W^i_k \cdot x) \bigg{)} \notag \\
&=& g_{k}^N(x)  + \frac{1}{\sqrt{N}} \sum_{i=1}^N \bigg{(} ( C^i_{k+1} - C^i_k )  \sigma( W^i_{k+1} \cdot x)  + (  \sigma( W^i_{k+1} \cdot x) -   \sigma( W^i_k \cdot x)  ) C^i_k \bigg{)} \notag \\
&=& g_{k}^N(x)  + \frac{1}{\sqrt{N}} \sum_{i=1}^N \bigg{(} ( C^i_{k+1} - C^i_k )  \bigg{[} \sigma(W^i_k \cdot x) +  \sigma'( W^{i,\ast}_{k} \cdot x_k ) x\cdot ( W^i_{k+1} - W^i_{k} ) \bigg{]}   \notag \\
&+& \bigg{[} \sigma'(W^i_k \cdot x) x\cdot ( W^i_{k+1} - W^i_k )   +  \frac{1}{2}\sigma''( W^{i, \ast \ast}_{k+1} \cdot x) \left( ( W^{i}_{k+1} - W^{i}_{k} ) \cdot x\right)^2  \bigg{]} C^i_k \bigg{)},
\label{gEvolution1}
\end{eqnarray}
for points $W^{i,\ast}_{k}$ and $W^{i,\ast,\ast}_{k}$ in the line segment connecting the points $W^{i}_{k}$ and $W^{i}_{k+1}$. Let $\alpha^N = \frac{\alpha}{N}$. Substituting (\ref{SGDupdates}) into (\ref{gEvolution1}) yields
\begin{eqnarray}
g_{k+1}^N(x)  &=& g_{k}^N(x) + \frac{\alpha}{N^2} \sum_{i=1}^N ( y_k - g_k^N(x_k) ) \sigma(W^i_k \cdot x_k ) \sigma(W^i_k \cdot x)   \notag \\
&+& \frac{\alpha}{N^2} \sum_{i=1}^N \sigma'(W^i_k \cdot x)   ( y_k - g_k^N(x_k) )  \sigma'(W^i_k \cdot x_k ) x_k \cdot x \big{(} C^i_k \big{)}^2 + \mathcal{O}(N^{-3/2}).
\label{gEvolution1b}
\end{eqnarray}

We can re-write the evolution of $g_k^N(x)$ in terms of the empirical measure $\nu_k^N$.
\begin{eqnarray}
g_{k+1}^N(x)  &=& g_{k}^N(x) + \frac{\alpha}{N} ( y_k - g_k^N(x_k) ) \la \sigma(w \cdot x_k ) \sigma(w \cdot x), \nu_k^N \ra   \notag \\
&+& \frac{\alpha}{N}   ( y_k - g_k^N(x_k) ) x_k\cdot x   \la \sigma'(w \cdot x)   \sigma'(w \cdot x_k ) c^2, \nu_k^N \ra + \mathcal{O}(N^{-3/2}).
\label{gEvolution2}
\end{eqnarray}

Using (\ref{gEvolution2}), we can write the evolution of $h_t^N$ for $t \in [0,T]$ as
\begin{eqnarray}
h_t^N &=& h_0^N + \sum_{k=0}^{\floor*{N t}-1} (  g_{k+1}^N - g_k ) \notag \\
&=& h_0^N +\frac{\alpha}{N}  \sum_{k=0}^{\floor*{N t}-1}  ( y_k - g_k^N(x_k) ) \la \sigma(w \cdot x_k ) \sigma(w \cdot x), \nu_k^N \ra   \notag \\
&+&  + \frac{\alpha}{N}   \sum_{k=0}^{\floor*{N t}-1}  ( y_k - g_k^N(x_k) ) x_k\cdot x   \la \sigma'(w \cdot x)   \sigma'(w \cdot x_k ) c^2, \nu_k^N \ra \notag \\
&+& \mathcal{O}(N^{-1/2}) \notag
\end{eqnarray}

Next, we decompose the summations into a drift and martingale component.
\begin{eqnarray}
h_t^N &=&  h_0^N + \frac{\alpha}{N}  \sum_{k=0}^{\floor*{N t}-1}  \int_{\mathcal{X} \times \mathcal{Y}} ( y - g_k^N(x') ) \la \sigma(w \cdot x' ) \sigma(w \cdot x), \nu_k^N \ra  \pi(dx',dy) \notag \\
&+& \frac{\alpha}{N}   \sum_{k=0}^{\floor*{N t}-1}  \int_{\mathcal{X} \times \mathcal{Y}} ( y - g_k^N(x') ) x\cdot x'    \la \sigma'(w \cdot x)   \sigma'(w \cdot x' ) c^2, \nu_k^N \ra  \pi(dx', dy) \notag \\
&+& \frac{\alpha}{N}  \sum_{k=0}^{\floor*{N t}-1} \bigg{(} ( y_k - g_k^N(x_k) ) \la \sigma(w \cdot x_k ) \sigma(w \cdot x), \nu_k^N \ra  - \int_{\mathcal{X} \times \mathcal{Y}}  ( y - g_k^N(x') ) \la \sigma(w \cdot x' ) \sigma(w \cdot x), \nu_k^N \ra  \pi(dx',dy)  \bigg{)}  \notag \\
&+& \frac{\alpha}{N}   \sum_{k=0}^{\floor*{N t}-1}  \bigg{(}  ( y_k - g_k^N(x_k) ) x_k\cdot x   \la \sigma'(w \cdot x)   \sigma'(w \cdot x_k ) c^2, \nu_k^N \ra  \notag \\
&-&  \int_{\mathcal{X} \times \mathcal{Y}} ( y - g_k^N(x') ) x\cdot x'    \la \sigma'(w \cdot x)   \sigma'(w \cdot x' ) c^2, \nu_k^N \ra  \pi(dx', dy) \bigg{)} \notag \\
&+& \mathcal{O}(N^{-1/2}) \notag
\end{eqnarray}

For convenience, we define the martingale terms (the third and fourth terms in the equation above) as $M_t^{N,1}$ and $M_t^{N,2}$, respectively. The equation for $h_t^N$ can be re-written in terms of a Riemann integral and the scaled measure $\mu_t^N$, yielding
\begin{eqnarray}
h_t^N &=&  h_0^N + \alpha \int_0^t  \int_{\mathcal{X} \times \mathcal{Y}}  ( y - h_s^N(x') ) \la \sigma(w \cdot x' ) \sigma(w \cdot x), \mu_s^N \ra  \pi(dx',dy) ds \notag \\
&+& \alpha \int_0^t \int_{\mathcal{X} \times \mathcal{Y}} ( y - h_s^N(x') ) x\cdot x'    \la \sigma'(w \cdot x)   \sigma'(w \cdot x' ) c^2, \mu_s^N \ra  \pi(dx', dy) ds  \notag \\
&+& M_t^{N,1} + M_t^{N,2} + \mathcal{O}(N^{-1/2}).
\label{hEvolutionWithRemainderTerms}
\end{eqnarray}

In addition, using conditional independence of the terms in the series for $M_t^{N,1}$ and $M_t^{N,2}$ as well as the bounds from Lemmas \ref{GLemmaBound} and \ref{CandWbounds}, we have that
\begin{eqnarray}
\mathbb{E} \bigg{[} \big{(} M_t^{N,1} \big{)}^2 \bigg{]} &\leq& \frac{K}{N}, \notag \\
\mathbb{E} \bigg{[} \big{(} M_t^{N,2} \big{)}^2 \bigg{]} &\leq& \frac{K}{N}.\notag
\end{eqnarray}

We can also analyze the evolution of the empirical measure $\nu_k^N$ in terms of test functions $f \in C^2_b(\mathbb{R}^{1 + d })$. Using a Taylor expansion, we find that

\begin{eqnarray}
\la f , \nu^N_{k+1} \ra - \la f , \nu^N_k \ra 
&=&  \frac{1}{N} \sum_{i=1}^N \bigg{(} f(C^i_{k+1}, W^i_{k+1} ) -  f(C^i_{k}, W^i_{k} )  \bigg{)} \notag \\
&=& \frac{1}{N} \sum_{i=1}^N \partial_c f(C^i_{k}, W^i_{k} ) ( C^i_{k+1} -  C^i_{k} )  + \frac{1}{N} \sum_{i=1}^N \nabla_w  f(C^i_{k}, W^i_{k} )^{\top}  ( W^i_{k+1} -  W^i_{k} ) \notag \\
&+& \frac{1}{N} \sum_{i=1}^N \partial^{2}_{c} f(\bar C^i_{k},  \bar W^i_{k} ) ( C^i_{k+1} -  C^i_{k} )^2  + \frac{1}{N} \sum_{i=1}^N ( C^i_{k+1} -  C^i_{k} )\nabla_{cw}  f(\hat C^i_{k}, \hat W^i_{k} )\cdot ( W^i_{k+1} -  W^i_{k} )    \notag \\
&+& \frac{1}{N} \sum_{i=1}^N ( W^i_{k+1} -  W^i_{k} )\cdot\nabla^{2}_{w} f(\tilde C^i_{k}, \tilde W^i_{k} ) ( W^i_{k+1} -  W^i_{k} ),
\label{NuEvolution1}
\end{eqnarray}
for points $(\bar C^{i}_{k}, \bar W^{i}_{k})$, $(\hat C^{i}_{k}, \hat W^{i}_{k})$ and $(\tilde C^{i}_{k}, \tilde W^{i}_{k})$ in the segments connecting $C^i_{k+1}$ with $C^i_{k}$ and  $W^i_{k+1}$ with $W^i_{k}$, respectively.

Substituting (\ref{SGDupdates}) into (\ref{NuEvolution1}) yields
\begin{eqnarray}
\la f , \nu^N_{k+1} \ra - \la f , \nu^N_k \ra &=&  N^{-5/2} \sum_{i=1}^N \partial_c f(C^i_{k}, W^i_{k} )  \alpha (y_k - g_k^N(x_k) )  \sigma (W^i_k \cdot  x_k)   \notag \\
&+& N^{-5/2} \sum_{i=1}^N   \alpha (y_k - g_{k}^N(x_k) )  C^i_k \sigma' (W^i_k \cdot x_k) \nabla_w  f(C^i_{k}, W^i_{k} )\cdot x_{k} + O_{p}\left(N^{-2}\right) \notag \\
&=&  N^{-3/2} \alpha (y_k - g_k^N(x_k) )   \la \partial_c f(c, w )   \sigma (w \cdot  x_k), \nu_k^N \ra   \notag \\
&+& N^{-3/2} \alpha (y_k - g_{k}^N(x_k) )     \la c \sigma' (w  \cdot x_k) \nabla_w  f(c,w )\cdot x_{k}, \nu_k^N \ra + O_{p}\left(N^{-2}\right).\nonumber
\end{eqnarray}

Therefore,
\begin{eqnarray}
\la f, \mu^N_t \ra &=&  \la f, \mu^N_0 \ra  + \sum_{k=0}^{\floor*{N t}-1} \bigg{(} \la f , \nu^N_{k+1} \ra - \la f , \nu^N_k \ra \bigg{)} \notag \\
&=&  \la f, \mu^N_0 \ra  +   N^{-3/2}  \sum_{k=0}^{\floor*{N t}-1}  \alpha (y_k - g_k^N(x_k) )   \la \partial_c f(c, w )   \sigma (w  \cdot x_k), \nu_k^N \ra   \notag \\
&+& N^{-3/2}  \sum_{k=0}^{\floor*{N t}-1}   \alpha (y_k - g_{k}^N(x_k) )     \la c \sigma' (w  \cdot x_k) \nabla_w  f(c,w )\cdot x_{k}, \nu_k^N \ra   + O_{p}\left(N^{-1}\right).
\label{muEvolutionWithRemainderTerms}
\end{eqnarray}

\section{Relative Compactness}  \label{RelativeCompactness}

In this section we prove that the family of processes $\{  \mu^N, h^N \}_{N}$ is relatively compact. Section \ref{CompactContainment} proves compact containment. Section \ref{Regularity} proves regularity. Section \ref{ProofOfRelativeCompactness} combines these results to prove relative compactness.

\subsection{Compact Containment} \label{CompactContainment}

We first establish a priori bounds for the parameters $(C_k^i, W_k^i)$.

\begin{lemma} \label{CandWbounds}

For all $i\in\mathbb{N}$ and all $k$ such that $k/N\leq T$,

\begin{eqnarray}
| C_k^i | &<& C < \infty \notag \\
\mathbb{E} \norm{ W_{k}^{i} } &<& C < \infty.\nonumber
\end{eqnarray}
\end{lemma}
\begin{proof}

The unimportant finite constant $C<\infty$ may change from line to line. We first observe that
\begin{eqnarray*}
| C_{k+1}^i  | &\leq&   | C_{k}^i | + \alpha N^{-3/2} \left| y_k - g_{k}^N(x_k)  \right|  | \sigma (W^i_k \cdot  x_k) | \notag \\
&\leq&  | C_{k}^i | + \frac{  C  | y_k |  }{N^{3/2}} +  \frac{C}{N^2} \sum_{i=1}^N | C_k^i |,
\end{eqnarray*}
where the last inequality follows from the definition of $g_{k}^N(x)$ and the uniform boundedness assumption on $\sigma(\cdot)$.

Then, we subsequently obtain that
\begin{eqnarray*}
| C_{k}^i | &=& | C_{0}^i |  + \sum_{j = 1}^k \bigg{[} | C_{j}^i | - | C_{j-1}^i | \bigg{]} \notag \\
&\leq& | C_{0}^i |  + \sum_{j=1}^k  \frac{ C   }{N^{3/2}}  +  \frac{C}{N^2} \sum_{j=1}^k \sum_{i=1}^N | C_{j-1}^i | \notag \\
&\leq& | C_{0}^i |  + \frac{C}{\sqrt{N}}  +  \frac{C}{N^2} \sum_{j=1}^k \sum_{i=1}^N | C_{j-1}^i |.
\label{Cbound0011}
\end{eqnarray*}

This implies that
\begin{eqnarray*}
\frac{1}{N} \sum_{i=1}^N | C_{k}^i | &\leq& \frac{1}{N} \sum_{i=1}^N | C_{0}^i |  +  \frac{C}{\sqrt{N}} +  \frac{C}{N^2} \sum_{j=1}^k \sum_{i=1}^N | C_{j-1}^i |,
\end{eqnarray*}

Let us now define $m_{k}^{N}=\frac{1}{N} \displaystyle \sum_{i=1}^N | C_{k}^i |$. Since the random variables $C_0^i$ take values in a compact set, we have that $ \frac{1}{N} \displaystyle \sum_{i=1}^N | C_{0}^i |  +  \frac{C}{\sqrt{N}}  < C < \infty$. Then,
\begin{eqnarray*}
m_{k}^{N} &\leq& C +  \frac{C}{N} \sum_{j=1}^k m_{j-1}^{N}.
\end{eqnarray*}

By the discrete Gronwall lemma and using $k/N\leq T$,
\begin{eqnarray}
m_k^N \leq C \exp \bigg{(} \frac{C k}{N} \bigg{)} \leq C.
\label{mBound}
\end{eqnarray}

Note that the constants may depend on $T$.

 We can now combine the bounds (\ref{mBound}) and (\ref{Cbound0011}) to yield, for any $0 \leq k \leq TN$,
\begin{eqnarray}
| C_{k}^i | &\leq& | C_{0}^i |  + \frac{C}{\sqrt{N}}  +  \frac{C}{N^2} \sum_{j=1}^k  m_{j-1}^N \notag \\
&\leq& | C_{0}^i |  + \frac{C}{\sqrt{N}}  +  \frac{C}{N^2} \sum_{j=1}^k  C_2 \notag \\
&\leq& | C_{0}^i |  + \frac{C}{\sqrt{N}}  +  \frac{C}{N} \notag \\
&\leq& C,
\label{Cbound0022}
\end{eqnarray}
where the last inequality follows from the random variables $C_0^i$ taking values in a compact set.

Now, we turn to the bound for $\parallel W^i_k \parallel$. We start with the bound (using Young's inequality)
\begin{align}
 \parallel W^i_{k+1} \parallel &\leq  \parallel W^i_k \parallel +  \frac{C}{N^{3/2}} \left( |y_k| +  \frac{1}{\sqrt{N} }\sum_{j=1}^N | C^j_k | \right) | C^i_k|    | \sigma'(W^i_k \cdot x_k ) | \parallel x_k \parallel\nonumber\\
 &\leq \parallel W^i_k \parallel +   C \left(\frac{1}{N} |y_k|^{2} + \frac{1}{N^{2}}\sum_{j=1}^N | C^j_k |^{2} + \frac{1}{N} | C^i_k|^{2}   \parallel x_k \parallel^{2}\right)\nonumber\\
 &\leq \parallel W^i_k \parallel +   C \left(\frac{1}{N} |y_k|^{2} + \frac{1}{N^{2}}\sum_{j=1}^N | C^j_k |^{2}+ \frac{1}{N} | C^i_k|^{4}+  \frac{1}{N} \parallel x_k \parallel^{4}\right),\nonumber
\end{align}
for a constant $C<\infty$ that may change from line to line. Taking an expectation, using Assumption \ref{A:Assumption1}, the bound (\ref{Cbound0022}), and using the fact that $k/N\leq T$, we obtain
\begin{align*}
\mathbb{E}\parallel W^i_k \parallel\leq C<\infty, 
\end{align*}
for all $i\in\mathbb{N}$ and all $k$ such that $k/N\leq T$, concluding the proof of the lemma.

\end{proof}

Using the bounds from Lemma \ref{CandWbounds}, we can now establish a bound for $g_k^N(x)$ for $x \in \mathcal{D}$.

\begin{lemma} \label{GLemmaBound}

For all $i\in\mathbb{N}$, all $k$ such that $k/N\leq T$, and any $x \in \mathcal{D}$,

\begin{eqnarray}
\mathbb{E} \bigg{[} | g_k^N(x) |^2 \bigg{]} &<& C < \infty.\nonumber
\end{eqnarray}

\end{lemma}
\begin{proof}

Recall equation (\ref{gEvolution1b}), which describes the evolution of $g_k^N(x)$.
\begin{eqnarray}
g_{k+1}^N(x)  &=& g_{k}^N(x) + \frac{\alpha}{N^2} \sum_{i=1}^N ( y_k - g_k^N(x_k) ) \sigma(W^i_k \cdot x_k ) \sigma(W^i_k \cdot x)   \notag \\
&+&   \frac{\alpha}{N^2} \sum_{i=1}^N \sigma'(W^i_k \cdot x)   ( y_k - g_k^N(x_k) )  \sigma'(W^i_k \cdot x_k ) x_k\cdot x \big{(} C^i_k \big{)}^2 + \frac{C}{N^{3/2}}. \notag
\end{eqnarray}

This leads to the bound
\begin{eqnarray}
| g_{k+1}^N(x) |  &\leq& | g_{k}^N(x) | + \frac{\alpha}{N^2} \sum_{i=1}^N | y_k - g_k^N(x_k) |   +  \frac{\alpha}{N^2} \sum_{i=1}^N  | y_k - g_k^N(x_k) |  \big{(} C^i_k \big{)}^2 + \frac{C}{N^{-3/2}} \notag \\
&\leq& | g_{k}^N(x) | + \frac{C}{N} | g_k^N(x_k) |  + \frac{C}{N}.\nonumber
\end{eqnarray}

We now square both sides of the above inequality.
\begin{eqnarray}
| g_{k+1}^N(x) |^2  &\leq& \big{(} | g_{k}^N(x) | + \frac{C}{N} | g_k^N(x_k) |  + \frac{C}{N} \big{)}^2 \notag \\
&\leq& | g_{k}^N(x) |^2 +2  | g_{k}^N(x) | (  \frac{C}{N} | g_k^N(x_k) |  + \frac{C}{N} ) +(  \frac{C}{N} | g_k^N(x_k) |  + \frac{C}{N} )^2 \notag \\
&\leq& | g_{k}^N(x) |^2 + \frac{C}{N}   | g_{k}^N(x) |^2  + \frac{C}{N},\nonumber
\end{eqnarray}
where the last line uses Young's inequality.

Therefore,
\begin{eqnarray}
| g_{k+1}^N(x) |^2 - | g_{k}^N(x) |^2 &\leq& \frac{C}{N} | g_k^N(x_k) |^2  + \frac{C}{N}.\nonumber
\end{eqnarray}

Then, using a telescoping series,

\begin{eqnarray}
| g_{k}^N(x) |^2 &=& | g_{0}^N(x) |^2 + \sum_{j=1}^k \bigg{(} | g_{j}^N(x) |^2 - | g_{j-1}^N(x) |^2 \bigg{)}  \notag \\
&\leq&  | g_{0}^N(x) |^2 + \sum_{j=1}^k  \bigg{(}  \frac{C}{N} | g_{j-1}^N(x_{j-1}) |^2 + \frac{C}{N} \bigg{)} \notag \\
&\leq&  | g_{0}^N(x) |^2 + C +  \frac{C}{N}  \sum_{j=1}^k  | g_{j-1}^N(x_{j-1}) |^2.\nonumber
\end{eqnarray}

Taking expectations,
\begin{eqnarray}
\mathbb{E} \bigg{[} | g_{k}^N(x) |^2  \bigg{]} \leq \mathbb{E} \bigg{[} | g_{0}^N(x) |^2  \bigg{]} + C +  \frac{C}{N}  \sum_{j=1}^k   \mathbb{E} \bigg{[} | g_{j-1}^N(x_{j-1}) |^2 \bigg{]}.\nonumber
\end{eqnarray}

Taking advantage of the fact that $x_j$ is sampled from a fixed dataset $\mathcal{D}$ of $M$ data samples,
\begin{eqnarray}
\mathbb{E} \bigg{[} | g_{k}^N(x) |^2  \bigg{]} &\leq& \mathbb{E} \bigg{[} | g_{0}^N(x) |^2  \bigg{]} + C +  \frac{C}{N}  \sum_{j=1}^k  \sum_{x' \in \mathcal{D}}  \mathbb{E} \bigg{[} | g_{j-1}^N(x') |^2 \bigg{]},
\label{gBound0011}
\end{eqnarray}
and therefore
\begin{eqnarray}
\sum_{x \in \mathcal{D}} \mathbb{E} \bigg{[} | g_{k}^N(x) |^2  \bigg{]} &\leq& \sum_{x \in \mathcal{D}} \mathbb{E} \bigg{[} | g_{0}^N(x) |^2  \bigg{]} + M C +  \frac{C M}{N}  \sum_{j=1}^k  \sum_{x' \in \mathcal{D}}  \mathbb{E} \bigg{[} | g_{j-1}^N(x') |^2 \bigg{]} \notag \\
&\leq& \sum_{x \in \mathcal{D}} \mathbb{E} \bigg{[} | g_{0}^N(x) |^2  \bigg{]} +  C +  \frac{C}{N}  \sum_{j=1}^k  \sum_{x \in \mathcal{D}}  \mathbb{E} \bigg{[} | g_{j-1}^N(x) |^2 \bigg{]}.
\label{gSumBound}
\end{eqnarray}

Recall that
\begin{eqnarray}
g_0^N(x) = \frac{1}{\sqrt{N}} \sum_{i=1}^N C^i_0 \sigma(W^i_0 \cdot x),\nonumber
\end{eqnarray}
where $(C^i_0, W^i_0)$ are i.i.d., mean-zero random variables. Then,
\begin{eqnarray}
\mathbb{E} \bigg{[} | g_0^N(x)  |^2 \bigg{]}   &\leq& \mathbb{E} \bigg{[} \bigg{(} \frac{1}{\sqrt{N}} \sum_{i=1}^N C^i_0 \sigma(W^i_0 \cdot x) \bigg{)}^2 \bigg{]} \notag \\
&\leq& \frac{C}{N} \sum_{i=1}^N \mathbb{E} \bigg{[} ( C^i_0 )^2 \bigg{]} \notag \\
&\leq& C.\nonumber
\end{eqnarray}

Combining this bound with the bound (\ref{gSumBound}) and using the discrete Gronwall lemma yields, for any $0 \leq k \leq TN$,
\begin{eqnarray}
\sum_{x \in \mathcal{D}} \mathbb{E} \bigg{[} | g_{k}^N(x) |  \bigg{]} \leq C.\nonumber
\end{eqnarray}

Substituting this bound into equation (\ref{gBound0011}) produces the desired bound
\begin{eqnarray}
\mathbb{E} \bigg{[} | g_{k}^N(x) |^2  \bigg{]} \leq C,\nonumber
\end{eqnarray}
for any $0 \leq k \leq TN$.

\end{proof}

We now prove compact containment for process $\{ (\mu_t^N, h_t^N), t \in [0,T]\}_{N\in\mathbb{N}}$. Recall that $(\mu_t^N, h_t^N) \in D_E([0,T])$ where $E = \mathcal{M}(\mathbb{R}^{1+d}) \times \mathbb{R}^M$.

\begin{lemma}\label{L:CompactContainment}
For each $\eta > 0$, there is a compact subset $\mathcal{K}$ of E such that
\begin{eqnarray*}
\sup_{N \in \mathbb{N}, 0 \leq t \leq T} \mathbb{P}[ (\mu_t^N, h_t^N) \notin \mathcal{K} ] < \eta.
\end{eqnarray*}
\end{lemma}
\begin{proof}
For each $L>0$, define $K_L=[-L,L]^{1+d}$.  Then, we have that $K_L$ is a compact  subset of $\mathbb{R}^{1+d}$, and for each $t\geq 0$ and $N\in \mathbb{N}$,
\begin{equation*}
\mathbb{E}\left[\mu^N_t(\mathbb{R}^{1+d}\setminus K_L)\right] = \frac{1}{N}\sum_{i=1}^N \mathbb{P}\left[ |c^i_{\floor*{N t}}|+\parallel w^i_{\floor*{N t}} \parallel \geq L\right] \leq \frac{C}{L}.
\end{equation*}
where we have used Markov's inequality and the bounds from Lemma \ref{CandWbounds}. We define the compact subsets of $\mathcal{M} ( \mathbb{R}^{1+d})$
\begin{equation*}
\hat{K}_L  = \overline{\left\{ \nu:\, \nu(\mathbb{R}^{1+d}\setminus K_{(L+j)^2}) < \frac{1}{\sqrt{L+j}} \textrm{ for all } j\in \mathbb{N}\right\}}
\end{equation*}
and we observe that
\begin{align*}
 \mathbb{P}\left\{ \mu^N_t\not \in \hat{K}_L\right] &\leq \sum_{j=1}^\infty \mathbb{P}\left[ \mu^N_t(\mathbb{R}^{1+d}\setminus K_{(L+j)^2} )> \frac{1}{\sqrt{L+j}}\right]
\leq \sum_{j=1}^\infty \frac{\mathbb{E}[\mu^N_t(\mathbb{R}^{1+d}\setminus K_{(L+j)^2})]}{1/\sqrt{L+j}}\\
&\le \sum_{j=1}^\infty \frac{C}{(L+j)^2/\sqrt{L+j}}
\le \sum_{j=1}^\infty \frac{C}{(L+j)^{3/2}}.
\end{align*}
Given that $\lim_{L\to \infty}\sum_{j=1}^\infty\frac{C}{(L+j)^{3/2}} =0$, we have that, for each $\eta > 0$, there exists a compact set $\hat{K}_L$ such that

\begin{eqnarray}
\sup_{N \in \mathbb{N}, 0 \leq t \leq T} \mathbb{P}[ \mu_t^N \notin \hat{K}_L ] < \frac{\eta}{2}.\nonumber
\end{eqnarray}

Due to Lemma \ref{GLemmaBound} and Markov's inequality, we also know that, for each $\eta > 0$, there exists a compact set $U = [-B, B]^{M}$ such that
\begin{eqnarray}
\sup_{N \in \mathbb{N}, 0 \leq t \leq T} \mathbb{P}[  h_t^N \notin U]  < \frac{\eta}{2}.\nonumber
\end{eqnarray}

Therefore, for each $\eta > 0$, there exists a compact set $ \hat K_L \times [-B, B]^{M} \subset E$ such that
\begin{eqnarray}
\sup_{N \in \mathbb{N}, 0 \leq t \leq T} \mathbb{P}[  (\mu_t^N, h_t^N) \notin \hat K_L \times [-B, B]^{M}]  < \eta.\nonumber
\end{eqnarray}

\end{proof}

\subsection{Regularity} \label{Regularity}

We now establish regularity of the process $\mu^{N}$ in $D_{\mathcal{M}(\mathbb{R}^{1+d})}([0,T])$. Define the function $q(z_{1},z_{2})=\min\{|z_{1}-z_{2}|,1\}$ where $z_{1},z_{2} \in \mathbb{R}$. Let $\mathcal{F}^N_t$ be the $\sigma-$algebra generated by $\{(C_{0}^{i},W_{0}^{i})\}_{i=1}^{N}$ and $\{x_{j}\}_{j=0}^{\floor{Nt}-1}$.

\begin{lemma}\label{MU:regularity}
Let  $f \in C^{2}_{b}(\mathbb{R}^{1+d})$. For any $\delta \in (0,1)$, there is a constant $C<\infty$ such that for $0\leq u\leq \delta$,  $0\leq v\leq \delta\wedge t$, $t\in[0,T]$,
\begin{equation*}
 \mathbb{E}\left[q(\left< f,\mu^N_{t+u}\right>,\left< f,\mu^N_t\right>)q(\left< f,\mu^N_t\right>,\left< f,\mu^N_{t-v}\right>)\big| \mathcal{F}^N_t\right]  \le  C \delta + \frac{C}{N^{3/2}}.
\end{equation*}
%
%
\end{lemma}
\begin{proof}
We start by noticing that a Taylor expansion gives for $0\leq s\leq t\leq T$ \begin{eqnarray}
| \la f , \mu^N_{t} \ra -  \la f , \mu^N_{s} \ra  | &=& | \la f,   \nu^N_{\floor*{N t} } \ra -  \la f,   \nu^N_{\floor*{N s} } \ra | \notag \\
&\leq& \frac{1}{N} \sum_{i=1}^N | f( C^i_{\floor*{N t}}, W^i_{\floor*{N t}}) - f(C^i_{\floor*{N s}}, W^i_{\floor*{N s}}) | \notag \\
&\leq& \frac{1}{N} \sum_{i=1}^N  | \partial_c f (  \bar C^i_{\floor*{N t}}, \bar W^i_{\floor*{N t}} ) | |  C^i_{\floor*{N t}} -  C^i_{\floor*{N s}} | \notag \\
&+& \frac{1}{N} \sum_{i=1}^N  \parallel \nabla_w f (  \hat C^i_{\floor*{N t}}, \hat W^i_{\floor*{N t}} ) \parallel \parallel  W^i_{\floor*{N t}} -  W^i_{\floor*{N s}} \parallel,
\label{Regularity1}
\end{eqnarray}
for points $(\bar C^{i}, \bar W^{i})$ and $(\hat C^{i}, \hat W^{i})$ in the segments connecting $C^i_{\floor*{N s}}$ with $C^i_{\floor*{N t}}$ and  $W^i_{\floor*{N s}}$ with $W^i_{\floor*{N t}}$, respectively.

Let's now establish a bound on $|  C^i_{\floor*{N t}} -  C^i_{\floor*{N s}} |$ for $s < t \leq T$ with $0<t-s\leq \delta<1$.
\begin{eqnarray}
 \mathbb{E} \bigg{[} |  C^i_{\floor*{N t}} -  C^i_{\floor*{N s}} | \bigg| \mathcal{F}^N_s   \bigg{]} &=& \mathbb{E} \bigg{[} | \sum_{k = \floor*{N s } }^{\floor*{N t}-1} ( C_{k+1}^i - C_k^i  ) |  \bigg| \mathcal{F}^N_s   \bigg{]}  \notag \\
&\leq&  \mathbb{E} \bigg{[} \sum_{k = \floor*{N s } }^{\floor*{N t}-1}  | \alpha (y_k - g_{k}^N(x_k) ) \frac{1}{N^{3/2}} \sigma (W^i_k \cdot x_k) |  \bigg| \mathcal{F}^N_s   \bigg{]}  \notag \\
&\leq& \frac{1}{N^{3/2}} \sum_{k = \floor*{N s } }^{\floor*{N t}-1} C  \leq \frac{C}{\sqrt{N}} (t -s ) + \frac{C}{N^{3/2}} \notag \\
&\leq& \frac{C}{\sqrt{N}} \delta + \frac{C}{N^{3/2}},
\label{CregularityBound00}
\end{eqnarray}
where Assumption \ref{A:Assumption1} was used as well as the bounds from Lemmas \ref{CandWbounds} and \ref{GLemmaBound}.

Let's now establish a bound on $\parallel  W^i_{\floor*{N t}} -  W^i_{\floor*{N s}} \parallel$ for $s < t \leq T$ with $0<t-s\leq \delta<1$. We obtain
\begin{eqnarray*}
  \mathbb{E} \bigg{[} \parallel W^i_{\floor*{N t}} -  W^i_{\floor*{N s}} \parallel \bigg| \mathcal{F}^N_s   \bigg{]}  &=&  \mathbb{E} \bigg{[} \parallel \sum_{k = \floor*{N s } }^{\floor*{N t}-1} ( W_{k+1}^i - W_k^i  ) \parallel \bigg| \mathcal{F}^N_s   \bigg{]}   \notag \\
&\leq&   \mathbb{E}  \bigg{[} \sum_{k = \floor*{N s } }^{\floor*{N t}-1} \parallel    \alpha (y_k - g_{k}^N(x_k) ) \frac{1}{N^{3/2}} C^i_k \sigma' (W^i_k \cdot x_k) x_k \parallel \bigg| \mathcal{F}^N_s   \bigg{]}   \notag \\
&\leq&  \frac{1}{N^{3/2}} \sum_{k = \floor*{N s } }^{\floor*{N t}-1} C  \notag \\
& \leq & \frac{C}{\sqrt{N}} (t -s ) + \frac{C}{N} \leq \frac{C}{\sqrt{N}} \delta + \frac{C}{N^{3/2}},
\label{WregularityBound00}
\end{eqnarray*}
where we have again used the bounds from Lemmas \ref{CandWbounds} and \ref{GLemmaBound}.

Now, we return to equation (\ref{Regularity1}). Due to Lemma \ref{CandWbounds}, the quantities $( \bar c^i_{\floor*{N t}}, \bar w^i_{\floor*{N t}} )$ are bounded in expectation for $0 < s < t \leq T$.  Therefore, for $0 < s < t \leq T$ with $0<t-s\leq \delta<1$
\begin{eqnarray*}
\mathbb{E}\left[| \la f , \mu^N_{t} \ra -  \la f , \mu^N_{s} \ra  | \big| \mathcal{F}^N_s \right] \leq C \delta + \frac{C}{N^{3/2}}.
\end{eqnarray*}
where $C<\infty$ is some unimportant constant. Then, the statement of the Lemma follows.
\end{proof}

We next establish regularity of the process $h^{N}_t$ in $D_{\mathbb{R}^M}([0,T])$. For the purposes of the following lemma, let the function $q(z_{1},z_{2})=\min\{ \norm{z_{1}-z_{2}},1\}$ where $z_{1},z_{2} \in \mathbb{R}^M$ and $\norm{z} = |z_1 | + \cdots +  |z_M|$.  \\

\begin{lemma}\label{H:regularity}
For any $\delta \in (0,1)$, there is a constant $C<\infty$ such that for $0\leq u\leq \delta<1$,  $0\leq v\leq \delta\wedge t$, $t\in[0,T]$,

\begin{equation*}
 \mathbb{E}\left[q(h_{t+u}^N, h_t^N )q(h_t^N,h_{t-v}^N )\big| \mathcal{F}^N_t \right]  \le  C \delta + \frac{C}{N}.
\end{equation*}

%
%
\end{lemma}
\begin{proof}

Recall that
\begin{eqnarray}
g_{k+1}^N(x)  = g_{k}^N(x)  + \frac{1}{\sqrt{N}} \sum_{i=1}^N  \bigg{(} C^i_{k+1} - C^i_k )  \sigma(W^i_{k+1} \cdot x)   +  \sigma'(W^{i,\ast}_k \cdot x) x\cdot ( W^i_{k+1} - W^i_k )   C^i_k \bigg{)}.\nonumber
\end{eqnarray}

Therefore,
\begin{eqnarray}
h_t^N(x) - h_s^N(x) &=& g_{\floor*{N t}}(x) - g_{\floor*{N s}}(x) \notag \\
&=& \sum_{k= \floor*{N s}}^{\floor*{N t}} ( g_{k+1}^N(x)  -  g_{k}^N(x) ) \notag \\
&=& \sum_{k= \floor*{N s}}^{\floor*{N t}}  \frac{1}{\sqrt{N}} \sum_{i=1}^N  \bigg{(} C^i_{k+1} - C^i_k )  \sigma(W^i_{k+1} \cdot x)   +  \sigma'(W^{i,\ast}_k \cdot x) x\cdot ( W^i_{k+1} - W^i_k )   C^i_k \bigg{)}.\nonumber
\end{eqnarray}

This yields the bound
\begin{eqnarray}
| h_t^N(x) - h_s^N(x)  | &\leq&  \sum_{k= \floor*{N s}}^{\floor*{N t}} |  g_{k+1}^N(x)  -  g_{k}^N(x) | \notag \\
&\leq& \sum_{k= \floor*{N s}}^{\floor*{N t}} \frac{1}{\sqrt{N}} \sum_{i=1}^N  \bigg{(} | C^i_{k+1} - C^i_k |    + \norm{ W^i_{k+1} - W^i_k } \bigg{)},\nonumber
\end{eqnarray}
where we have used the boundedness of $\sigma'(\cdot)$ (from Assumption \ref{A:Assumption1}) and the bounds from Lemma \ref{CandWbounds}.

Taking expectations,
\begin{eqnarray}
\mathbb{E} \bigg{[} | h_t^N(x) - h_s^N(x)  | \bigg{|} \mathcal{F}_s^N \bigg{]} &\leq&   \frac{1}{\sqrt{N}} \sum_{i=1}^N   \sum_{k= \floor*{N s}}^{\floor*{N t}} \mathbb{E} \bigg{[} | C^i_{k+1} - C^i_k |    + \norm{ W^i_{k+1} - W^i_k } \bigg{|} \mathcal{F}_s^N \bigg{]}.\nonumber
\end{eqnarray}

Using the bounds (\ref{CregularityBound00}) and (\ref{WregularityBound00}),
\begin{eqnarray}
\mathbb{E} \bigg{[} | h_t^N(x) - h_s^N(x)  | \bigg{|} \mathcal{F}_s^N \bigg{]} &\leq&   \frac{1}{\sqrt{N}} \sum_{i=1}^N \bigg{(} \frac{C}{\sqrt{N}}(t-s) + \frac{C}{N^{3/2}} \bigg{)} \notag \\
&=&  C (t -s ) + \frac{C}{N}.
\label{GregularityBound00}
\end{eqnarray}

The bound (\ref{GregularityBound00}) holds for each $x \in \mathcal{D}$. Therefore,
\begin{eqnarray}
\mathbb{E} \bigg{[} \norm{h_t^N - h_s^N} \bigg{|} \mathcal{F}_s^N \bigg{]} \leq C (t -s ) + \frac{C}{N}.\nonumber
\end{eqnarray}

The statement of the Lemma then follows.

\end{proof}

\subsection{Combining our results to prove relative compactness} \label{ProofOfRelativeCompactness}

\begin{lemma}\label{L:RelativeCompactness}
The family of processes $\{\mu^N, h^N \}_{N\in\mathbb{N}}$ is relatively compact in $D_{E}([0,T])$.
\end{lemma}

\begin{proof}
Combining Lemmas \ref{L:CompactContainment} and \ref{MU:regularity}, and Theorem 8.6 of Chapter 3 of \cite{EthierAndKurtz} proves that $\{\mu^N\}_{N\in\mathbb{N}}$ is relatively compact in $D_{\mathcal{M}(\mathbb{R}^{1+d})}([0,T])$. (See also Remark 8.7 B of Chapter 3 of \cite{EthierAndKurtz} regarding replacing $\sup_N$ with $\lim_N$ in the regularity condition B of Theorem 8.6.)

Similarly, combining Lemmas \ref{L:CompactContainment} and \ref{H:regularity} proves that $\{h^N\}_{N\in\mathbb{N}}$ is relatively compact in $D_{\mathbb{R}^{M}}([0,T])$.

Since relative compactness is equivalent to tightness, we have that the probability measures of the family of processes $\{\mu^N\}_{N\in\mathbb{N}}$ are tight. Similarly, we have that the probability measures of the family of process $\{ h^N \}_{N\in\mathbb{N}}$ are tight. Therefore, $\{\mu^N, h^N \}_{N\in\mathbb{N}}$ is tight. Then, $\{\mu^N, h^N \}_{N\in\mathbb{N}}$ is also relatively compact.

\end{proof}


\section{Identification of the Limit} \label{Identification}

Let $\pi^N$ be the probability measure of a convergent subsequence of $\left(\mu^N, h^N \right)_{0\leq t\leq T}$. Each $\pi^N$ takes values in the set of probability measures $\mathcal{M} \big{(} D_E([0,T]) \big{)}$. Relative compactness, proven in Section \ref{RelativeCompactness}, implies that there is a subsequence $\pi^{N_k}$ which weakly converges. We must prove that any limit point $\pi$ of a convergent subsequence $\pi^{N_k}$ will satisfy the evolution equation (\ref{EvolutionEquationIntroductionXavier}). \\
\begin{lemma}
Let $\pi^{N_k}$ be a convergent subsequence with a limit point $\pi$. Then, $\pi$ is a Dirac measure concentrated on $(\mu, h) \in D_E([0,T])$ and $(\mu, h)$ satisfies equation (\ref{EvolutionEquationIntroductionXavier}).
\end{lemma}
\begin{proof}
We define a map $F(\mu,h): D_{E}([0,T]) \rightarrow \mathbb{R}_{+}$ for each $t \in [0,T]$, $f \in C^{2}_{b}(\mathbb{R}^{1+d})$, $g_{1},\cdots,g_{p}\in C_{b}(\mathbb{R}^{1+d})$, $q_{1},\cdots,q_{p}\in C_{b}(\mathbb{R}^{M})$, and $0\leq s_{1}<\cdots< s_{p}\leq t$.
\begin{eqnarray}
F(\mu, h) &=&  \bigg{|} \left(\la f, \mu_t \ra - \la f,  \mu_0 \ra \right) \times\la g_{1},\mu_{s_{1}}\ra\times\cdots\times \la g_{p},\mu_{s_{p}}\ra\bigg{|} \notag \\
&+& \sum_{x \in \mathcal{D}}  \bigg{|} \bigg{(}  h_t(x)  - h_0(x) -  \alpha \int_0^t \int_{\mathcal{X} \times \mathcal{Y}} ( y - h_s(x') ) \la \sigma(w \cdot x ) \sigma(w \cdot x'), \mu_s \ra  \pi(dx',dy) ds \notag \\
&-& \alpha \int_0^t \int_{\mathcal{X} \times \mathcal{Y}} ( y - h_s(x') )   \la c^2  \sigma'(w \cdot x')  \sigma'(w \cdot x) x\cdot x', \mu_s \ra \pi(dx,dy) ds \bigg{)} \times q_{1}(h_{s_{1}})  \times\cdots\times q_{p}(h_{s_{p}}) \bigg{|} \nonumber
\end{eqnarray}

Then, using equations (\ref{hEvolutionWithRemainderTerms}) and (\ref{muEvolutionWithRemainderTerms}), we obtain

\begin{eqnarray}
\mathbb{E}_{\pi^N} [ F(\mu, h ) ] &=& \mathbb{E} [ F( \mu^N, h^N) ] \notag \\
&=& \mathbb{E} \left|  \mathcal{O}_p(N^{-1/2})  \times \prod_{i=1}^{p} \la g_{i},\mu^{N}_{s_{i}}\ra \right|   \notag \\
&+&  \mathbb{E} \left| (  M_t^{N,1} + M_t^{N,2} + \mathcal{O}_p(N^{-1/2}) ) \times \prod_{i=1}^{p} q_{i}( h^{N}_{s_{i}}) \right|   \notag \\.
&\leq& C \bigg{(} \mathbb{E} \bigg{[} | M^{1,N}(t) |^2  \bigg{]}^{\frac{1}{2}}+  \mathbb{E} \bigg{[}| M^{2,N}(t)  |^2  \bigg{]}^{\frac{1}{2} } \bigg{)}+ O(N^{-1/2})   \notag \\
&\leq& C \left(\frac{1}{\sqrt{N}}+\frac{1}{N}\right), \notag
\end{eqnarray}
where we have used the Cauchy-Schwarz inequality.

Therefore,
\begin{eqnarray}
\lim_{N \rightarrow \infty} \mathbb{E}_{\pi^N} [ F(\mu, h) ] = 0.\notag
\end{eqnarray}
Since $F(\cdot)$ is continuous and $F( \mu^N)$ is uniformly bounded (due to the uniform boundedness results of Section \ref{RelativeCompactness}),
\begin{eqnarray}
 \mathbb{E}_{\pi} [ F(\mu, h) ] = 0.\notag
\end{eqnarray}
Since this holds for each $t \in [0,T]$, $f \in C^{2}_{b}(\mathbb{R}^{1+d})$ and $g_{1},\cdots,g_{p}, q_1, \cdots, q_p \in C_{b}(\mathbb{R}^{1+d})$, $(\mu, h)$ satisfies the evolution equation (\ref{EvolutionEquationIntroductionXavier}).
\end{proof}

\section{Proof of Convergence} \label{ProofOfConvergence}
We now combine the previous results of Sections \ref{RelativeCompactness} and \ref{Identification} to prove Theorem \ref{MainTheorem1}. Let $\pi^N$ be the probability measure corresponding to $(\mu^N, h^N)$. Each $\pi^N$ takes values in the set of probability measures $\mathcal{M} \big{(} D_E([0,T]) \big{)}$. Relative compactness, proven in Section \ref{RelativeCompactness}, implies that every subsequence $\pi^{N_k}$ has a further sub-sequence $\pi^{N_{k_m}}$ which weakly converges. Section \ref{Identification} proves that any limit point $\pi$ of $\pi^{N_{k_m}}$ will satisfy the evolution equation (\ref{EvolutionEquationIntroductionXavier}). Equation (\ref{EvolutionEquationIntroductionXavier}) is a finite-dimensional, linear equation and therefore has a unique solution. Therefore, by Prokhorov's Theorem, $\pi^N$ weakly converges to $\pi$, where $\pi$ is the distribution of $(\mu, h)$, the unique solution of (\ref{EvolutionEquationIntroductionXavier}). That is, $(\mu^N, h^N)$ converges in distribution to $(\mu, h)$.

\section{Proof of Corollary \ref{PositiveDefinite} } \label{ProofOFPositiveDefinite}

This section proves that under reasonable hyperparameter choices, the matrix $A$ in the limit equation will be positive definite.

\begin{proof}

We first show that $A$ is equivalent to the covariance matrix of the random variables $U =  \bigg{(} U(x^{(1)}), \ldots,  U(x^{(M)}) \bigg{)}$, which are defined as

\begin{eqnarray}
U(x) &=& \sqrt{ \frac{\alpha}{M} } \sigma(W \cdot x ) +  \sqrt{ \frac{\alpha}{M}  }   C  \sigma'(W \cdot x ) x 
\end{eqnarray}
where $(W, C) \sim \mu_0$ and $x \in \mathcal{D}$. 
%
%
Due to the fact that $C$ is a mean zero random variable and independent of $W$, we have
\begin{eqnarray}
\mathbb{E} \bigg{[} U(x) U(x') \bigg{]} =  \mathbb{E} \bigg{[} \frac{\alpha}{M}  \sigma(W \cdot x )  \sigma(W \cdot x')+   \frac{\alpha}{M}     C^2  \sigma'(W \cdot x )\sigma'(W \cdot x' ) x\cdot x' \bigg{]} = A_{x, x'}.
\end{eqnarray}

To prove that $A$ is positive definite, we need to show that $z^{\top} A z > 0$ for every non-zero $z \in \mathbb{R}^M$.
\begin{eqnarray}
z^{\top} A z &=& z^{\top}  \mathbb{E} \bigg{[} U U^{\top} \bigg{]}  z \notag \\
&=& \mathbb{E} \bigg{[} ( z^{\top} U )^2 \bigg{]} \notag \\
&=& \frac{\alpha}{M} \mathbb{E} \bigg{[} \bigg{(}  \sum_{i=1}^M z_i\left( \sigma(x^{(i)} \cdot W ) +   C  \sigma'(W \cdot x^{(i)} ) x^{(i)}\right)  \bigg{)}^2 \bigg{]}.
\end{eqnarray}

The functions $\sigma(x^{(i)} \cdot W )$ are linearly independent since the $x^{(i)}$ are in district directions (see Remark 3.1 of \cite{yIto}). Therefore, for each non-zero $z$, there exists a point $w^{\ast}$
such that
\begin{eqnarray}
 \sum_{i=1}^M z_i \sigma(x^{(i)} \cdot w^{\ast} )  \neq 0.\nonumber
\end{eqnarray}

Consequently, there exists an $\epsilon > 0$ such that
\begin{eqnarray}
\bigg{(}  \sum_{i=1}^M z_i \sigma(x^{(i)} \cdot w^{\ast} )  \bigg{)}^2 > \epsilon.\nonumber
\end{eqnarray}

Since $\sigma(w\cdot x)+ c\sigma'(w\cdot x)x$ is a continuous function, there exists a set $B = \{ (c,w) : \norm{w - w^{\ast} } +\norm{c}< \eta \}$ for some $\eta > 0$ such that for $(c,w) \in B$
\begin{eqnarray}
\bigg{(}  \sum_{i=1}^M z_i \left(\sigma(x^{(i)} \cdot w )  +   C  \sigma'(W \cdot x^{(i)} ) x^{(i)} \right)\bigg{)}^2 > \frac{\epsilon}{2}.\nonumber
\end{eqnarray}

Then,
\begin{eqnarray}
 \mathbb{E} \bigg{[} \bigg{(}  \sum_{i=1}^M z_i \left(\sigma(x^{(i)} \cdot W )+   C  \sigma'(W \cdot x^{(i)} ) x^{(i)}  \right)\bigg{)}^2 \bigg{]} &\geq&  \mathbb{E} \bigg{[} \bigg{(}  \sum_{i=1}^M z_i \left(\sigma(x^{(i)} \cdot W ) +   C  \sigma'(W \cdot x^{(i)} ) x^{(i)}\right) \bigg{)}^2 \mathbf{1}_{W \in B} \bigg{]} \notag \\
 &\geq& \mathbb{E} \bigg{[} \frac{\epsilon}{2} \mathbf{1}_{(C,W) \in B} \bigg{]} \notag \\
 &=& \frac{\epsilon}{2} K,\nonumber
\end{eqnarray}
for some constant $K > 0$.

Therefore, for every non-zero $z \in \mathbb{R}^M$,
\begin{eqnarray}
z^{\top} A z > 0,\nonumber
\end{eqnarray}
and $A$ is positive definite, concluding the proof of the Corollary.
\end{proof}


\begin{thebibliography}{99}



\bibitem
{yIto}
Yoshifusa Ito.
\newblock Nonlinearity creates linear independence.
\newblock Advances in Computational Mathematics,
\newblock 5: 189-203, 1996.




\bibitem
{Xavier}
X. Glorot and Y. Bengio.
\newblock Understanding the difficulty of training deep feedforward neural networks.
\newblock Proceedings of the thirteenth international conference on artificial intelligence and statistics, pp. 249-256. 2010.

\bibitem
{JasonLee}
S. Du, J. Lee, H. Li, L. Wang, and X. Zhai.
\newblock Gradient Descent Finds Global Minima of Deep Neural Networks.
\newblock Proceedings of the 36th International Conference on Machine Learning, Long Beach, California, PMLR 97, 2019.

\bibitem
{Du1}
S. Du, X. Zhai, B. Poczos, and A. Singh.
\newblock Gradient Descent Provably Optimizes Over-Parameterized Neural Networks.
\newblock ICLR, 2019.

\bibitem
{UCLA2018}
D. Zou, Y. Cao, D. Zhou, and Q. Gu.
\newblock Stochastic Gradient Descent Optimizes Over-parameterized Deep ReLU Networks.
\newblock arXiv: 1811.08888, 2018.

\bibitem
{NTK}
A. Jacot, F. Gabriel, and C. Hongler.
\newblock Neural Tangent Kernel: Convergence and Generalization in Neural Networks.
\newblock 32nd Conference on Neural Information Processing Systems (NeurIPS 2018), Montreal, Canada.


\bibitem
{EthierAndKurtz}
S. Ethier and T. Kurtz.
\newblock Markov Processes: Characterization and Convergence.
\newblock 1986,
\newblock Wiley, New York, MR0838085.




%




\bibitem
{Hornik1}
K. Hornik, M. Stinchcombe, and H. White.
\newblock Multilayer feedforward networks are universal approximators.
\newblock Neural Networks,
\newblock 2(5), 359-366, 1989.

\bibitem
{Hornik2}
K. Hornik.
\newblock Approximation capabilities of multilayer feedforward networks.
\newblock Neural Networks,
\newblock 4(2), 251-257, 1991.


\bibitem
{Hornik3}
C. Kuan and K. Hornik.
\newblock Convergence of learning algorithms with constant learning rates.
\newblock IEEE Transactions on Neural Networks,
\newblock 2(5), 484-489, 1991.



\bibitem
{Montanari}
S. Mei, A. Montanari, and P. Nguyen.
\newblock A mean field view of the landscape of two-layer neural networks
\newblock Proceedings of the National Academy of Sciences, 115 (33) E7665-E767, 2018.

\bibitem
{SirignanoSpiliopoulosNN1}
J. Sirignano and K. Spiliopoulos.
\newblock Mean Field Analysis of Neural Networks: A Law of Large Numbers.
\newblock SIAM Journal on Applied Mathematics,  Vol. 80, Issue 2, pp. 725--752, 2020.

\bibitem
{SirignanoSpiliopoulosNN2}
J. Sirignano and K. Spiliopoulos.
\newblock Mean Field Analysis of Neural Networks: A Central Limit Theorem.
\newblock Stochastic Processes and their Applications, Volume 130, Issue 3, March 2020, pp. 1820-1852, 2020.

\bibitem
{SirignanoSpiliopoulosNN3}
J. Sirignano and K. Spiliopoulos.
\newblock Mean Field Analysis of Deep Neural Networks.
\newblock Mathematics of Operations Research, 2021, to appear.

\bibitem
{RVE}
G. M. Rotskoff and E. Vanden-Eijnden.
\newblock Neural Networks as Interacting Particle Systems: Asymptotic Convexity of the Loss Landscape and Universal Scaling of the Approximation Error.
\newblock arXiv:1805.00915, 2018.






\end{thebibliography}
\end{document}